\documentclass[11pt]{amsart}
\usepackage{amsmath,amsthm,hyperref,enumitem,mathtools,esint}
\usepackage[alphabetic,initials,nobysame]{amsrefs}
\usepackage{color,todonotes}
\usepackage{tikz-cd}

\title[Definitions of quasiconformality]{Definitions of quasiconformality on metric surfaces}

\author{Damaris Meier}
\address[Damaris Meier]{Department of Mathematics\\ University of Fribourg\\ Chemin du Mus\'ee 23\\ 1700 Fribourg, Switzerland.}
\email{damaris.meier@unifr.ch}

\author{Kai Rajala}
\address[Kai Rajala]{Department of Mathematics and Statistics, University of Jyväskylä, P.O. Box 35 (MaD), 40014 University of Jyväskylä, Finland}
\email{kai.i.rajala@jyu.fi}

\date{\today}

\numberwithin{equation}{section}

\newtheorem{thm}{Theorem}[section]

\newtheorem{lemma}[thm]{Lemma}
\newtheorem{corollary}[thm]{Corollary}
\newtheorem{proposition}[thm]{Proposition}

\theoremstyle{remark}

\theoremstyle{definition}

\newcommand{\R}{\mathbb{R}}
\newcommand{\N}{\mathbb{N}}

\newcommand{\Nloc}{N^{1,2}_{\text{loc}}}
\newcommand{\Lloc}{L_{\text{loc}}}
\newcommand{\hm}{{\mathcal H}}

\newcommand{\diam}{\operatorname{diam}}

\renewcommand{\mod}{\operatorname{Mod}}

\DeclareMathOperator{\md}{md}
\DeclareMathOperator{\ap}{ap}
\DeclareMathOperator{\apmd}{ap\,md}
\newcommand{\jac}{\mathrm{Jac}}

\begin{document}
\begin{abstract}
We explore the interplay between different definitions of distortion for mappings $f\colon X\to \mathbb{R}^2$, where $X$ is any metric surface, meaning that $X$ is homeomorphic to a domain in $\mathbb{R}^2$ and has locally finite 2-dimensional Hausdorff measure. We establish that finite distortion in terms of the familiar analytic definition always implies finite distortion in terms of \emph{maximal and minimal stretchings along paths}. The converse holds for maps with locally integrable distortion. In particular, we prove the equivalence of various notions of quasiconformality, implying a novel uniformization result for metric surfaces.
\end{abstract}
\maketitle

\section{Introduction}
Within this note we study the relation between different notions of distortion for mappings on metric surfaces. Here, a \emph{metric surface} $X$ is a metric space homeomorphic to a domain in $\R^2$ with locally finite $2$-dimensional Hausdorff measure. Most importantly, we show that locally integrable \emph{distortion along paths} introduced by the authors in \cite{MR23} is comparable to the analytic distortion for mappings $f:X \to \R^2$. 

Before stating the main theorem, we provide the relevant definitions. Let $X$ and $Y$ be metric surfaces and consider the Newton-Sobolev space $\Nloc(X,Y)$, see Section \ref{section:Sobolev}. We call a map $f:X \to Y$ \emph{sense-preserving} if for any domain $\Omega$ compactly contained in $X$ so that $f|_{\partial\Omega}$ is continuous it follows that $\deg(y,f,\Omega)\geq1$ for any $y\in f(\Omega)\setminus f(\partial\Omega)$. Here, $\deg$ is the local topological degree of $f$ (see \cite{Ric93}*{I.4}).

Let $f\in\Nloc(X,Y)$ be sense-preserving. We say that $f$ has \emph{finite distortion along paths} and denote $f \in \operatorname{FDP}(X,Y)$ if there is a measurable $K\colon X\to[1,\infty)$ such that 
\begin{equation} \label{ineq:path-distortion} 
\rho_f^u(x)\leq K(x)\cdot \rho_f^l(x) \quad \text{for almost every } x\in X,
\end{equation}
where $\rho_f^u$ and $\rho_f^l$ denote the minimal weak upper and maximal weak lower gradient of $f$, respectively; for definitions see Section \ref{section:Sobolev}. The \emph{distortion along paths} $K_f$ of $f$ is 
$$K_f(x):=\begin{cases}
\frac{\rho_f^u(x)}{\rho_f^l(x)}, & \text{if } \rho_f^l(x)\neq0,\\
1, & \text{if } \rho_f^l(x)=0.
\end{cases} $$

We say that $f$ has \emph{finite analytic distortion}, denoted 
$f\in\operatorname{FDA}(X,Y)$, if there is a measurable $C\colon X\to[1,\infty)$ such that 
\begin{equation}\label{ineq:analytic-distortion}
\rho^u_f(x)^2 \leq C(x)\cdot J_f(x) \quad \text{for almost every } x\in X,
\end{equation} 
where 
\begin{equation}
\label{area_limsup}
J_f(x)=\limsup_{r \to 0} \frac{\hm^2_Y(f(\overline{B}(x,r)))}{\pi r^2}. 
\end{equation} 
If $f$ is a homeomorphism and $X$ is a domain in $\R^2$ or $Y=\R^2$, then $J_f$ coincides with the Radon-Nikodym derivative of the corresponding pull-back measure with respect to $\hm_X^2$, see Corollary \ref{cor:jacohomeo}. Notice, however, that such a pull-back is not defined for non-homeomorphic maps. 

The \emph{analytic distortion} $C_f$ of $f$ is 
$$C_f(x):=\begin{cases}
\frac{\rho_f^u(x)^2}{J_f(x)}, & \text{if } J_f(x)\neq0,\\
1, & \text{if } J_f(x)=0.
\end{cases} $$

Inequality \eqref{ineq:analytic-distortion} is equivalent to \eqref{ineq:path-distortion} whenever $f$ is a map between euclidean domains. However, in the generality of metric spaces, it is unclear how the two definitions relate. The following main theorem of this work shows equivalence for mappings from a metric surface into $\R^2$.

\begin{thm}\label{thm:main}
    Let $f\in\Nloc(X,\R^2)$ be sense-preserving.
    \begin{enumerate}
        \item\label{thm:main-2} If $f \in \operatorname{FDA}(X,\R^2)$, then $f \in \operatorname{FDP}(X,\R^2)$ and 
        \begin{align*} 
           K_f(x) \leq 4\sqrt{2}\, C_f(x) \quad \text{for almost every } x\in X.
        \end{align*}
        
        \item\label{thm:main-1} If $f \in \operatorname{FDP}(X,\R^2)$ and $K_f\in\Lloc^1(X)$, then $f \in \operatorname{FDA}(X,\R^2)$ and 
        \begin{align*}
            C_f(x) \leq 4\sqrt{2}\, K_f(x) \quad \text{for almost every } x\in X.
        \end{align*} 
        
    \end{enumerate}
\end{thm}

We do not know if the second part holds without assumption $K_f \in \Lloc^1(X)$. 
As a consequence of Theorem \ref{thm:main} we obtain the following characterization of quasiconformal homeomorphisms. 

\begin{corollary}\label{cor:qc-equivalence}
    If $f\colon X\to f(X) \subset \R^2$ is a homeomorphism, then the following are quantitatively equivalent.
    \begin{enumerate}
        \item\label{cor:qc-eq-ana} $f$ is analytically quasiconformal,
        \item\label{cor:qc-eq-geo} $f$ is geometrically quasiconformal,
        \item\label{cor:qc-eq-path} $f$ is quasiconformal along paths.
    \end{enumerate} 
    Moreover, if $f$ satisfies any of the three conditions, then so does $f^{-1}$. 
\end{corollary}

Here a homeomorphism $f \in \Nloc(X,Y)$ is \emph{analytically quasiconformal} (resp., \emph{quasiconformal along paths}), if $C_f$ (resp., $K_f$) is uniformly bounded. Moreover, $f$ is \emph{geometrically quasiconformal} if there is $C\geq 1$ such that 
\begin{equation} \label{ineq:geomQC}
C^{-1}\cdot\mod \Gamma\leq \mod f(\Gamma) \leq C\cdot \mod \Gamma \end{equation} 
for each curve family $\Gamma$ in $X$, where $\mod$ refers to $2$-modulus, see Section \ref{section:modulus}, and $f(\Gamma)$ denotes the family of curves $f\circ\gamma$ for $\gamma\in\Gamma$. In order to prove the equivalence of Conditions (1) and (2), we apply Williams' theorem \cite{Wil:12}.

There is a large body of literature on different definitions of quasiconformality in metric spaces, showing in particular the equivalence of the \emph{metric definition} with the analytic and geometric definitions for homeomorphisms between metric spaces with \emph{controlled geometry}, see \cite{BKR07}, \cite{HeiKos95}, \cite{HeiKos98}, \cite{HKST01}, \cite{Tys98}, \cite{Tys01}. However, in the generality of metric surfaces the metric definition is not equivalent with the other definitions (see \cite{RRR21}*{Section 5}), and does not lead to a satisfactory theory. 

Lower gradients and the class $\operatorname{FDP}(X,Y)$ were introduced in \cite{MR23} as a tool for developing the fundamental properties of non-homeomorphic maps under minimal assumptions. In particular, we proved in \cite{MR23} that a non-constant $f \in \operatorname{FDP}(X,\R^2)$ with $K_f \in L^1_{\text{loc}}(X)$ is continuous, discrete and open. Non-homeomorphic maps with controlled distortion in metric spaces have previously been considered e.g. 
in \cite{Cri06}, \cite{Guo15}, \cite{Kir14}, \cite{OnnRaj09}.

Theorem \ref{thm:main} can be applied to the \emph{uniformization problem}  of metric surfaces (see e.g. \cites{BonKle02,Raj:17,LWarea,LWintrinsic,Iko:19,MW21,NR:21,Mei22}) as follows. Here $f \in \operatorname{FDP}(X,Y)$ is \emph{quasiregular}, or has \emph{bounded distortion} (along paths), if $K_f$ is uniformly bounded. 

\begin{thm}\label{thm:reciprocal} 
    If $X$ admits a non-constant quasiregular map $f:X \to \R^2$, then $X$ admits a quasiconformal homeomorphism $\phi:X \to U$ onto a domain $U\subset\R^2$.    
\end{thm}

Non-homeomorphic maps are easier to construct than homeomorphisms, so Theorem \ref{thm:reciprocal} offers flexibility for finding quasiconformal parametrizations of a given surface. Theorem \ref{thm:reciprocal} is sharp in the following sense: There is no $p \geq 1$ for which the existence of a non-constant $f \in \operatorname{FDP}(X,\R^2)$ with $K_f\in\Lloc^p(X)$ implies the existence of a quasiconformal homeomorphism $\phi:X \to U$ onto a domain $U\subset\R^2$, see \cite{MR23}*{Proposition 6.1}.

Theorem \ref{thm:main} also allows the extension of the classical  \emph{Sto\"ilow factorization theorem} (see \cite{AIM09}*{Chapter 5.5}, \cite{LP20}) to our setting. 

\begin{thm}\label{thm:factorization}
    Every non-constant quasiregular map $f : X \to \R^2$ admits a factorization $f=g\circ v$, where $v\colon X\to V$ is a quasiconformal homeomorphism onto a domain $V\subset\R^2$ and $g\colon V\to\R^2$ is complex analytic.
\end{thm}

\textbf{Acknowledgments.} Part of this research was conducted while the first named author was visiting University of Jyväskylä. She wishes to thank the department and staff for their hospitality.

\section{Preliminaries} \label{section:prel}

\subsection{Basic definitions and notations}
Let $(X,d)$ be a metric space. We denote the \emph{open ball} in $X$ of radius $r>0$ centered at a point $x\in X$ by $B(x,r)$. If $B=B(x,r)$ is a ball and $k>0$, we denote by $kB$ the ball $B(x,kr)$. We say that a subdomain $\Omega$ of $X$ is \emph{compactly contained} in $X$ if the closure $\overline{\Omega}$ is compact. Given a set $A \subset X$ and $\delta>0$, we denote the closed $\delta$-neighborhood of $A$ in $X$ by $N_\delta(A)$. 

The \emph{image} of a curve $\gamma$ in $X$ is indicated by $|\gamma|$ and the \emph{length} by $\ell(\gamma)$. A curve $\gamma$ is \emph{rectifiable} if $\ell(\gamma)<\infty$ and \emph{locally rectifiable} if each of its compact subcurves is rectifiable. If $\gamma\colon[a,b]\to X$ is rectifiable, then for almost every $t\in[a,b]$ we can define the \emph{metric differential} of $\gamma$ at $t$ by $$\gamma'(t):=\lim_{s\to t,\,s\neq t}\frac{d(\gamma(t),\gamma(s))}{|t-s|}.$$ A curve $\gamma\colon[0,\ell(\gamma)]\to X$ is \emph{parametrized by arclength} if $\ell(\gamma|_I)=|I|_1$ for every interval $I\subset[0,\ell(\gamma)]$. Here and later on, $|\cdot|_n$ denotes the \emph{$n$-dimensional Lebesgue measure}. 

For $s\geq 0$, we denote the \emph{$s$-dimensional Hausdorff measure} of a set $A\subset X$ by $\mathcal{H}^s(A)$ or $\mathcal{H}^s_X(A)$ if we want to emphasize that $A$ is a subset of $X$. The normalizing constant is chosen so that $|U|_n=\mathcal{H}^n(U)$ for open subsets $U$ of $\R^n$.

If $X$ is a metric surface, we equip $X$ with $\mathcal{H}^2$. Let $L^p(X)$ ($L^p_{\text{loc}}(X)$) denote the \emph{space of $p$-integrable (locally $p$-integrable) Borel functions} from $X$ to $\R\cup\{-\infty,\infty\}$. Here locally $p$-integrable means $p$-integrable on compact subsets.

\subsection{Modulus}\label{section:modulus}
Let $X$ be a metric surface and $\Gamma$ a family of curves in $X$. A Borel function $g\colon X \to [0,\infty]$ is \emph{admissible} for $\Gamma$ if $\int_{\gamma}g\, ds\geq 1$ for all locally rectifiable curves $\gamma\in \Gamma$. We define the ($2$-)\emph{modulus} of $\Gamma$ as 
$$\mod \Gamma = \inf_g \int_X g^2 \, d\mathcal H^2,$$
where the infimum is taken over all admissible functions $g$ for $\Gamma$. If there are no admissible functions for $\Gamma$ we set $\mod \Gamma = \infty$. A property is said to hold for \emph{almost every} curve in $\Gamma$ if it holds for every curve in $\Gamma\setminus\Gamma_0$ for some family $\Gamma_0\subset \Gamma$ with $\mod(\Gamma_0)=0$.

\subsection{Metric Sobolev spaces}\label{section:Sobolev}
Let $f\colon X\to Y$ be a map from metric surface $X$ to a metric space $Y$. A Borel function $\rho^u \colon X\to [0,\infty]$ is an \textit{upper gradient} of $f$ if 
\begin{align}\label{ineq:upper_gradient}
    d_Y(f(x),f(y)) \leq \int_{\gamma} \rho^u \, ds
\end{align}
for all $x,y\in X$ and every rectifiable curve $\gamma$ in $X$ joining $x$ and $y$. If the \textit{upper gradient inequality} \eqref{ineq:upper_gradient} holds for almost every rectifiable curve $\gamma$ in $X$ joining $x$ and $y$ we call $\rho^u$ \emph{weak upper gradient} of $f$. 

The Sobolev space $N^{1,2}(X,Y)$ is the space of Borel maps $f \colon X \to Y$ with upper gradient $\rho^u \in L^2(X)$ such that $x \mapsto d_Y(y,f(x))$ is in $L^2(X)$ for some $y \in Y$. The space $\Nloc(X, Y)$ is defined in the obvious manner. Each $f\in \Nloc(X,Y)$ has a \textit{minimal weak upper gradient} $\rho_f^u$, i.e.,\ for any other weak upper gradient $\rho^u$ we have $\rho_f^u\leq \rho^u$ almost everywhere. Moreover, $\rho^u_f$ is unique up to a set of measure zero, see \cite{HKST:15}*{Theorem 6.3.20}. We refer to the monograph \cite{HKST:15} for more background on metric Sobolev spaces.

\subsection{Metric differentiability}\label{sec:metric-diff}
Let $X$ be a metric surface and $U\subset\R^2$ a domain. We say that $u\colon U\to X$ is \emph{approximately metrically differentiable} at $z\in U$ if there exists a seminorm $N_z$ on $\R^2$ for which
    $$\ap\lim_{y\to z}\frac{d_X(u(y),u(z))-N_z(y-z)}{|y-z|}=0.$$
Here $\ap\lim$ denotes the approximate limit (see \cite{EG92}*{Section 1.7.2}). If such a seminorm exists, it is unique and is called \emph{approximate metric derivative} of $u$ at $z$, denoted $\ap\md u_z$. The \emph{Jacobian} of $\ap\md u_z$ is  $${J(\ap\md u_z)=\frac{\pi}{|B_z|_2}},$$ whenever $\ap\md u_z$ is a norm and $J(\ap\md u_z)=0$ otherwise. Here $B_z$ refers to the closed unit ball in $(\R^2,\apmd u_z)$. Every map $u\in \Nloc(U,X)$ is approximately metrically differentiable at almost every $z\in U$, see \cite{LWarea}*{Proposition 4.3}.

Let $U\subset\R^2$ be a domain and $u\in\Nloc(U,X)$. By \cite{HKST:15}*{Theorem 8.1.49}, $U$ is the union of pairwise disjoint Borel sets $G_j^u$, $j=0,1,\ldots$, so that $|G_0^u|_2=0$ and $u|_{G_j^u}$ is $j$-Lipschitz continuous for every $j \geq 1$. Recall the classical area formula following from \cite{Kar07}*{Theorem 3.2}. 

\begin{thm}[Classical area formula]\label{thm:classical-area-formula}
    For $u\in\Nloc(U,X)$ and every measurable set $A\subset U\setminus G_0^u$ we have
    \begin{align*}
       \int_A J(\ap\md u_z) \,dz=\int_X N(x,u,A)\,d\mathcal H^2.
    \end{align*}
\end{thm}
Here $N(x,u,A)$ denotes the number of preimages of $x$ under $u$ in $A$. If $u\in\Nloc(U,X)$ is a homeomorphism, Theorem \ref{thm:classical-area-formula} implies $J_u(z)=J(\apmd u_z)$ for almost every $z\in U$; recall the definition of $J_u(z)$ in \eqref{area_limsup}.

\subsection{Lower gradients and distortion along paths}
Let $X$ and $Y$ be metric surfaces. We call a Borel function $\rho^l\colon X\to[0,\infty]$ a \emph{lower gradient} of $f\in\Nloc(X,Y)$ if $\rho^l\leq\rho_f^u$ almost everywhere and
\begin{align}\label{ineq:lower_gradient}
    \ell(f\circ\gamma)\geq\int_{\gamma}\rho^l \, ds
\end{align}
for every rectifiable curve $\gamma$ in $X$ such that $f$ is continuous along $\gamma$. If the \textit{lower gradient inequality} \eqref{ineq:lower_gradient} holds for almost every rectifiable $\gamma$ on which $f$ is continuous, we call $\rho^l$ \emph{weak lower gradient} of $f$. Note that $0$ is always a lower gradient. Up to exceptional curve families of zero modulus, the upper gradient inequality \eqref{ineq:upper_gradient} is equivalent to the converse inequality in \eqref{ineq:lower_gradient}. Each $f\in \Nloc(X,Y)$ has a \emph{maximal weak lower gradient} $\rho_f^l$, i.e.,\ for any other weak lower gradient $\rho^l$ we have $\rho_f^l\geq\rho^l$ almost everywhere, that is uniquely defined up to sets of measure zero, see \cite{MR23}*{Section 7}.

Mappings of finite distortion along paths, i.e., class $\operatorname{FDP}(X,Y)$ (defined in the introduction), were introduced in \cite{MR23}. We now state the most important results from \cite{MR23} that will be repeatedly used throughout this work.

\begin{proposition}[\cite{MR23}*{Remarks 2.3 and 2.8}]\label{prop:sense-pres}
    Let $f\in\Nloc(X,\R^2)$ be sense-preserving. Then, $f$ is continuous and satisfies Lusin's condition (N). 
\end{proposition}
Here a map $f\colon X\to Y$ satisfies \emph{Lusin's condition (N)} if $\hm_Y^2(f(E))=0$ for every $E\subset X$ with $\hm_X^2(E)=0$. Recall that $f$ is \emph{discrete}, if $f^{-1}(y)$ is a discrete set in $X$ for every $y \in Y$. 

\begin{thm}[\cite{MR23}*{Theorem 1.2}]\label{thm:main-MR23}
    Let $f \in \operatorname{FDP}(X,\R^2)$ be non-constant with $K_f\in\Lloc^1(X)$. Then $f$ is open and discrete. 
\end{thm}

Let $U \subset \R^2$ be a domain. The \emph{maximal and minimal stretches} of a map $h\in\Nloc(U,Y)$ at points of approximate differentiability are defined by 
\begin{eqnarray*} 
L_h(z)=\max\{\ap\md h_z(v):|v|=1\}, \quad  
l_h(z)=\min\{\ap\md h_z(v):|v|=1\}. 
\end{eqnarray*}

\begin{lemma}[\cite{MR23}*{Lemma 2.9}] \label{lem:MR23-2.9}
Let $h \in \Nloc(U,Y)$. Then $L_h$ and $l_h$ are representatives of the minimal weak upper gradient and the maximal weak lower gradient of $h$, respectively. Moreover, 
\begin{equation*}
2^{-1} L_h(z)l_h(z) \leq J(\ap\md h_z) \leq 
2 L_h(z)l_h(z)  
\end{equation*} 
at points $z\in U$ of approximate differentiability.  
\end{lemma}

\subsection{Weakly quasiconformal parametrizations of metric surfaces}
The proof of Theorem \ref{thm:main} depends on the existence of a weakly quasiconformal parametrization of $X$ provided by \cite{NR22}. See also \cite{NR:21} and \cite{MW21}. The following theorem summarizes the main properties of a weakly quasiconformal parametrization and will be repeatedly applied within this work. A map $u\colon U\to X$ is called \emph{monotone}, if $u^{-1}(x)$ is connected for every $x\in X$.

\begin{thm}\label{thm:unif}
    If $X$ is a metric surface then there exists a continuous, surjective, sense-preserving and monotone map $u\in\Nloc(U,X)$, where $U$ is a domain in $\R^2$, such that
    \begin{enumerate}[label=(\roman*)]
        \item $u$ is $\sqrt{2}$-quasiregular. \label{thm:unif-1}
    \end{enumerate}
    Let $f\in\Nloc(X,Y)$ and $h:=f\circ u$. Then
    \begin{enumerate}[label=(\roman*)]
        \setcounter{enumi}{1}
        \item $h\in\Nloc(U,Y)$, and if $f$ is sense-preserving then so is $h$. \label{thm:unif-2}
    \end{enumerate}
    Moreover, if $f \in \operatorname{FDP}(X,Y)$, then
    \begin{enumerate}[label=(\roman*)]
        \setcounter{enumi}{2}
        \item $h \in \operatorname{FDP}(U,Y)$ with $K_h(z)\leq\sqrt{2}\,K_f(u(z))$ for almost every $z \in U$. \label{thm:unif-3} 
    \end{enumerate}
\end{thm}

\begin{proof}
    The existence of a sense-preserving weakly $(4/\pi)$-quasiconformal para\-met\-ri\-za\-tion $u\in\Nloc(U,X)$ follows from \cite{NR22}*{Theorem~1.3}. We refrain from defining weak quasiconformality here, but note that such a map $u$ is continuous, surjective and monotone, and satisfies
    \begin{align}\label{ineq:wqc}
        \mod \Gamma\leq \frac{4}{\pi}\mod u(\Gamma)
    \end{align}
    for every family $\Gamma$ of curves in $X$. 

    If $\ap\md u_z$ is a norm, then the closed unit ball $B_z$ of $(\R^2,\ap\md u_z)$ contains a unique ellipse of maximal area $E_z$, called \emph{John's ellipse} of $\ap\md u_z$. The proof of \cite{NR22}*{Theorem~1.3} implies that we may assume $E_z$ to be a disc for almost every $z\in U$. By John's Theorem (see e.g.\ \cite{Bal97}*{Theorem 3.1}) and Lemma \ref{lem:MR23-2.9} we know that
    \begin{align*}
        \rho^u_u(z)\leq\sqrt{2}\cdot \rho^l_u(z)
    \end{align*}
    holds for almost every $z\in U$. Thus $u$ is $\sqrt{2}$-quasiregular, which proves \ref{thm:unif-1}. 
    
    It follows from \eqref{ineq:wqc} that $u$ maps curve families of positive modulus to curve families of positive modulus. Therefore, $$\rho=\rho_u^u\cdot(\rho_f^u\circ u)$$ is a weak upper gradient of $h$, see Lemma \ref{lem:MR23-2.9} and \cite{MR23}*{Lemma 2.10}. By \cite{NR:21}*{Remark 7.2}, we know that $N(x,u,U)=1$ for almost every $x\in u(U)$. Let $E\subset U$ be compact and $G_0^u\subset U$ the exceptional set in the classical area formula, Theorem \ref{thm:classical-area-formula}. Combining the formula with \ref{thm:unif-1} and Lemma \ref{lem:MR23-2.9}, we have 
    \begin{align*}
        \int_E\rho^2\,dz&=\int_{E\setminus G_0}(\rho_u^u)^2\cdot(\rho_f^u\circ u)^2\,dz\leq\sqrt{2}\int_{E\setminus G_0}\rho_u^u\rho_u^l\cdot(\rho_f^u\circ u)^2\,dz\\
        &\leq 2\sqrt{2}\int_X(\rho_f^u)^2\cdot N(x,u,E)\,d\hm^2_X=2\sqrt{2}\int_{u(E)}(\rho_f^u)^2\,d\hm^2_X<\infty.
    \end{align*}
    In particular, $h\in\Nloc(X,Y)$. The second claim in \ref{thm:unif-2} follows from the basic properties of topological degree. 
    
    Finally by \cite{MR23}*{Lemma 2.10} and Lemma \ref{lem:MR23-2.9} we have 
    \begin{equation}
    \label{hanni}
    \rho^l_{h}(z)\geq \rho_f^l(u(z))\cdot \rho^l_u(z) 
    \quad \text{and} \quad 
    \rho^u_{h}(z)\leq \rho_f^u(u(z))\cdot \rho^u_u(z) 
    \end{equation}
    for almost every $z \in U$. Combining \eqref{hanni} with \ref{thm:unif-1} and \ref{thm:unif-2} gives \ref{thm:unif-3}.
\end{proof}

\subsection{Area inequality}
Another important ingredient in the proof of Theorem \ref{thm:main} is the following area inequality for Sobolev maps on metric surfaces.  

Let $X,Y$ be metric surfaces and $u\in\Nloc(U,X)$ a weakly quasiconformal parametrization as in Theorem \ref{thm:unif}. Given $f\in\Nloc(X,Y)$, let $G_0:=G_0^u\cup G_0^h$, where $G_0^u\subset U$ and $G_0^h\subset U$ are the exceptional sets in the classical area formula (Theorem \ref{thm:classical-area-formula}) associated with $u$ and $h=f \circ u$, respectively. We denote 
\begin{equation} \label{exceptionals}
u(G_0)=:X_0 \quad \text{and} \quad X \setminus X_0=:X'. 
\end{equation}

\begin{thm}[Area inequality, \cite{MR23}*{Theorem 3.1}]\label{thm:area-ineq}
    Let $f\in\Nloc(X,Y)$. If $g\colon Y\to [0,\infty]$ and $F\subset X'$ are Borel measurable, then
    \begin{align*}
        \int_F g(f(x))\cdot\rho_f^u(x)\rho_f^l(x)\,d\hm^2_X\leq 4\sqrt{2}\int_{Y}g(y)\cdot N(y,f,F)\,d\hm_Y^2.
    \end{align*}
    If $f$ additionally satisfies Lusin's condition (N), then
    \begin{align*}
        \int_F g(f(x))\cdot\rho_f^u(x)\rho_f^l(x)\,d\hm^2_X\geq \frac{1}{4\sqrt{2}}\int_{Y}g(y)\cdot N(y,f,F)\,d\hm_Y^2.
     \end{align*}
\end{thm}

\subsection{Covering theorems}
We recall the basic $5r$-covering lemma. For a proof see e.g. \cite{Hei01}*{Theorem 1.2}.

\begin{lemma}[$5r$-covering lemma]\label{lem:5r}
    Every family $\mathcal{F}$ of balls in $X$ of uniformly bounded diameter contains a subfamily $\mathcal{G}$ such that every two distinct balls in $\mathcal{G}$ are disjoint and 
    $$\bigcup_{B\in\mathcal{F}}B\subset\bigcup_{B\in\mathcal{G}}5B.$$
\end{lemma}

For a Borel function $g\colon X\to\R$, we define the maximal function 
    $$\mathcal{M}g(x)=\sup_{r>0}\frac{1}{\hm^2(B(x,5r))}\int_{B(x,r)}g\,d\hm^2.$$ 
The proof of the following lemma is a standard application of the $5r$-covering theorem, see e.g.\ \cite{Hei01}*{Theorem 2.2}

\begin{lemma}\label{lem:max-fct}
    If $g\in\Lloc^1(X)$ and $A\subset X$ with $\hm^2(A)>0$, then there are $E'\subset A$ with $\hm^2(A\setminus E')>0$ and $L\geq1$ such that $$\mathcal{M}g(x)\leq L\quad\text{for every }x\in A\setminus E'.$$
\end{lemma}

We will also apply the Vitali covering theorem for Hausdorff measures, see e.g. \cite{AmbTil04}*{Theorem 2.2.2}. 

\begin{thm} \label{thm:Vit}
Let $G \subset X$, and let $\mathcal{F}$ be a \emph{fine covering} of $G$ by closed sets. Then there exists a countable disjoint subfamily $\{V_j\} \subset \mathcal{F}$ such that one of the following holds: 
\begin{itemize}
\item[(i)] $\sum \diam(V_j)^2= \infty$. 
\item[(ii)] $\hm^2(G \setminus \cup_j V_j)=0$.  
\end{itemize}
\end{thm}
Here a covering $\mathcal{F}$ of $G$ by closed sets is \emph{fine} if for every $x \in G$ and every $\varepsilon>0$ there exists $V \in \mathcal{F}$ such that $x \in V$ and $\diam(V) < \varepsilon$. 

Lemma \ref{lem:5r} and Theorem \ref{thm:Vit} hold in arbitrary metric spaces, and the latter holds with exponent $2$ replaced by any $\alpha \geq 0$.  

\subsection{Regularity of the Hausdorff measure }
Let $X$ be a metric surface and $u\in\Nloc(U,X)$ a weakly quasiconformal parametrization as in Theorem \ref{thm:unif}. Moreover, let $G_0$, $X_0$ and $X'$ be as in \eqref{exceptionals}. As described in the paragraph preceding Theorem \ref{thm:classical-area-formula}, $U\setminus G_0$ may be covered with pairwise disjoint Borel sets $G_j^u\subset U$, $j=1,2,...$, so that $u|_{G_j^u}$ is $j$-Lipschitz. 
In particular, $X'=u(U\setminus G_0)$ is countably $2$-rectifiable. The following density result follows by combining \cite{Fed69}*{Theorem 2.10.19(5)} and \cite{Kir94}*{Theorem 9}. 

\begin{thm}\label{thm:regular}
    There exists $E\subset X$, $\hm^2(E)=0$, so that
    \begin{eqnarray*} 
& & \limsup_{r \to 0}\frac{\hm^2(\overline{B}(x,r))}{\pi r^2}\leq 1 \quad \text{for every } x \in X \setminus E, \quad \text{and} \\
& & \lim_{r \to 0}\frac{\hm^2(\overline{B}(x,r)\cap X')}{\pi r^2} = 1 \quad \text{for every } x \in X' \setminus E. 
    \end{eqnarray*}
\end{thm}

\section{Differentiation of Hausdorff measures}
Metric surfaces do not need to be doubling or even Vitali spaces, so standard results on differentiation of measures do not hold automatically. In this section we prove such results for sense-preserving Sobolev maps. 

Let $X$ be a metric surface. We fix a weakly quasiconformal parametrization $u:U \to X$ as above. 
Given $f\in\Nloc(X,\R^2)$, we denote $h=f\circ u$ and let $G_0$ and $X_0=u(G_0)$ be the exceptional sets in \eqref{exceptionals}. Recall notations $X'=X \setminus X_0$ and  
$$
J_f(x)=\limsup_{r \to 0} \frac{|f(\overline{B}(x,r))|_2}{\pi r^2}. 
$$
\begin{lemma} \label{lem:J_f=0}
    If $f \in \Nloc(X,\R^2)$ is sense-preserving, then $$J_f(x)=0\quad\text{for almost every }x\in X_0.$$
\end{lemma}

\begin{proof} 
Suppose towards contradiction that there are $\varepsilon>0$ and $W \subset X_0$ with $\hm^2_X(W)>0$ and such that $J_f(x) \geq 2\varepsilon$ for every $x \in W$. By choosing a subset if necessary, we may assume that $W$ is compact. We fix $\delta >0$. Then the collection of balls $B(x,r) \subset N_\delta(W)$ satisfying $x \in W$, $0<10r<\delta$, and 
$$
|f(\overline{B}(x,r))|_2 \geq \varepsilon \pi r^2
$$
covers $W$. By the $5r$-covering lemma (Lemma \ref{lem:5r}) there is a subcollection 
$\{B_j=\overline{B}(x_j,r_j)\}$ of disjoint closed balls so that collection $\{5B_j\}$ covers $W$. Then 
\begin{equation} \label{eq:aara}
\hm_{\delta,X}^2(W) \leq \sum_j 25\pi r_j^2 \leq 25 \varepsilon^{-1}
\sum_j |f(B_j)|_2. 
\end{equation} 
As before, we denote $h=f \circ u$ and recall that $h$ satisfies Condition (N) by Proposition \ref{prop:sense-pres} and Theorem \ref{thm:unif}. In particular, if we denote $F_j=u^{-1}(B_j)$ then the classical area formula (Theorem \ref{thm:classical-area-formula}) shows that 
\begin{equation} 
\label{eq:aara2} 
|f(B_j)|_2 \leq \int_{F_j} J(\ap\md h_z) \, dz \quad \text{for all } j. 
\end{equation} 
Since sets $F_j$ are pairwise disjoint, combining \eqref{eq:aara} and 
\eqref{eq:aara2} shows that 
$$
\hm_{\delta,X}^2(W) \leq 25 \varepsilon^{-1} \int_{u^{-1}(N_\delta(W))} J(\ap\md h_z) \, dz. 
$$
But $u^{-1}(W) \subset G_0$ has zero area, so the right hand term tends to zero when $\delta \to 0$. We conclude that $\hm^2_X(W)=0$, which is a contradiction. The proof is complete. 
\end{proof}

\begin{proposition}\label{prop:area-ineq-jac}
    If $f\in\Nloc(X,\R^2)$ is sense-preserving, then
    \begin{align}\label{ineq:area-ineq-jac}
        \int_FJ_f(x)\,d\hm_X^2\leq\int_{\R^2}N(y,f,F)\,dy
    \end{align}
    for every Borel set $F\subset X$. If $f$ is furthermore open and discrete, then equality holds in \eqref{ineq:area-ineq-jac}.
\end{proposition}

\begin{proof}
We start with the proof of \eqref{ineq:area-ineq-jac}. By Lemma \ref{lem:J_f=0} we may assume that $F \subset X' \setminus E$, where $E$ is the exceptional set in Theorem \ref{thm:regular}. 
Given $1<t<2$ and $k \in \mathbb{Z}$, denote 
$$
A_t^k=\{x \in F: \, t^{k-1} \leq J_f(x) < t^k\}. 
$$
Then \eqref{ineq:area-ineq-jac} follows from 
\begin{equation}
\label{ineq:aa1}
\int_{A_t^k} J_f(x) \, d\hm^2_X \leq t \int_{\R^2} N(y,f,A_t^k)\, dy 
\end{equation}
by summing both sides over $k$ and letting $t \to 1$. 

To prove \eqref{ineq:aa1}, we fix $t$ and $k$ and notice that it suffices 
to prove \eqref{ineq:aa1} for all compact subsets $A \subset A_t^k$. 
We fix such an $A$, and $\varepsilon>0$. 

Then, since $A \subset X' \setminus E$, Theorem \ref{thm:regular} and the definition of $A_t^k$ show that the collection 
$$
\mathcal{F}=\{\overline{B}(x,r): \, x \in A, \, 0<r<\varepsilon,\, \eqref{ineq:bb1} \text{ and } \eqref{ineq:bb2} \text{ hold } \} 
$$
is a fine covering of $A$; here we apply conditions  
\begin{equation} \label{ineq:bb1} 
(1+\varepsilon)^{-1} \pi r^2 \leq 
\hm^2(\overline{B}(x,r)\cap X')  \leq \hm^2(\overline{B}(x,r))  \leq  
(1+\varepsilon) \pi r^2, 
\end{equation}
and 
\begin{equation} \label{ineq:bb2} 
(1+\varepsilon)^{-1} t^{k-1}  \hm^2(\overline{B}(x,r)) \leq 
|f(\overline{B}(x,r))|_2
\leq (1+\varepsilon) t^k \hm^2(\overline{B}(x,r)). 
\end{equation}

Since $X$ is homeomorphic to $\mathbb{R}^2$ and $A$ is compact, we may choose $\varepsilon$ to be small enough so that $\hm^2(N_\varepsilon(A))<\infty$. Then \eqref{ineq:bb1} shows that 
if $\mathcal{G}$ is a subcollection of $\mathcal{F}$ consisting of pairwise disjoint balls $B_1,B_2,\ldots$, $B_j=\overline{B}(x_j,r_j)$, 
then 
$$
(1+\varepsilon)^{-1}\pi \sum_j r_j^2 
\leq \sum_j \hm^2(B_j) \leq \hm^2(N_\varepsilon(A)) < \infty. 
$$
Thus, by the Vitali covering theorem (Theorem \ref{thm:Vit}), the pairwise disjoint balls $B_j$ can be chosen so that 
\begin{equation} 
\label{eq:almost} 
\hm^2(A \setminus \cup_j B_j)=0. 
\end{equation}
Using \eqref{ineq:bb2} and \eqref{eq:almost}, we have 
\begin{eqnarray*} 
\int_A J_f(x) \, d\hm^2_X 
&\leq& t^k \hm^2(A) \leq t^k \sum_j \hm^2(B_j) 
\leq t(1+\varepsilon) \sum_j |f(B_j)|_2 \\
&\leq& t(1+\varepsilon) \sum_j \int_{\R^2} N(y,f,B_j) \, dy \\
&=& t(1+\varepsilon) \int_{\R^2}N(y,f,\cup_j B_j) \, dy \\
&\leq& t(1+\varepsilon) \int_{\R^2}N(y,f,N_\varepsilon(A)) \, dy. 
\end{eqnarray*}
By compactness of $A$, letting $\varepsilon \to 0$ yields \eqref{ineq:aa1} 
for $A_t^k$ replaced with $A$. Inequality \eqref{ineq:area-ineq-jac} follows. 

We now assume that $f$ is open and discrete and claim that 
 \begin{align}\label{ineq:area-ineq-jac_2}
        \int_FJ_f(x)\,d\hm_X^2\geq\int_{\R^2}N(y,f,F)\,dy
    \end{align}
for every Borel set $F \subset X$. Recall that $|f(X_0)|_2=0$ 
by Proposition \ref{prop:sense-pres} and Theorem \ref{thm:unif}. Therefore, we may again assume that $F \subset X' \setminus E$. 

Also, recall that an open and discrete map is locally invertible outside a discrete set $\mathcal{B}_f$. Therefore, we may replace $F$ with $F \setminus \mathcal{B}_f$ if needed and assume without loss of generality that $f$ is locally invertible at every $x \in F$. 

As in the proof of \eqref{ineq:area-ineq-jac}, we see that \eqref{ineq:area-ineq-jac_2} follows if we can show that 
\begin{equation}
\label{ineq:armak}
\int_{\R^2} N(y,f,A) \, dy \leq t \int_A J_f(x) \, d\hm^2_X 
\end{equation}
for every $1<t<2$, $k \in \mathbb{Z}$, and every compact set $A \subset A_t^k$. We can choose a family of pairwise disjoint balls $B_1,B_2,\ldots$ satisfying the conditions of collection $\mathcal{F}$ above, and require the additional property that $f_{|B_j}$ is invertible for each $j$. In particular, \eqref{eq:almost} holds and as $f$ satisfies Lusin's condition (N), by Proposition \ref{prop:sense-pres}, also $
|f(A \setminus \cup_j B_j)|_2=0$. Combining with \eqref{ineq:bb2}, we obtain 
\begin{eqnarray*} 
\int_{\R^2}N(y,f,A) \, dy 
&\leq& \int_{\R^2} N(y,f, \cup_j B_j) \, dy 
= \sum_j \int_{\R^2} N(y,f,B_j) \, dy \\
&=&\sum_{j} |f(B_j)|_2 
\leq (1+\varepsilon)t^k \sum_j \hm^2(B_j) \\
&\leq& (1+\varepsilon)t^k \hm^2(N_\varepsilon(A)). 
\end{eqnarray*}
Letting $\varepsilon \to 0$, the last term converges to 
$$
t^k \hm^2(A) \leq t \int_E J_f(x) \, d\hm^2_X. 
$$
Combining the estimates gives \eqref{ineq:armak}. The proof is complete. 
\end{proof}

Proposition \ref{prop:area-ineq-jac} together with Theorem \ref{thm:area-ineq} and Proposition \ref{prop:sense-pres} now imply the following corollary.
\begin{corollary} \label{cor:rect_dis}
    If $f \in \Nloc(X,\R^2)$ is sense-preserving, then 
 $$
J_f(x) \leq 
4\sqrt{2}\, \rho_f^u(x)\rho_f^l(x) \quad \text{for almost every } x \in X'.
$$ 
If $f$ is furthermore open and discrete, then $$
\rho_f^u(x)\rho_f^l(x) \leq 
4\sqrt{2}\, J_f(x) \quad \text{for almost every } x \in X'.
$$ 
\end{corollary}

\begin{corollary}\label{cor:jacohomeo} 
Let $f\in\Nloc(X,\R^2)$ be a homeomorphism and $\mu$ the corresponding pull-back measure, i.e., $\mu(A)=|f(A)|_2$ for all Borel sets $A \subset X$. Then $J_f$ is the Radon-Nikodym derivative of $\mu$ with respect to $\hm^2_X$.  
\end{corollary} 

\begin{proof} 
Recalling that $|f(X_0)|_2=0$, the claim follows from Proposition \ref{prop:area-ineq-jac} and the definition of Radon-Nikodym derivative. 
\end{proof}

\section{Proof of the main theorem}
The goal of this section is to prove Theorem \ref{thm:main}. Part \ref{thm:main-2} (from analytic distortion to distortion along paths) follows by combining Lemma \ref{lem:J_f=0} and Corollary \ref{cor:rect_dis}, and recalling that if $f \in \operatorname{FDA}(X,\R^2)$ then $\rho_f^u=0$ almost everywhere in the zero set of $J_f$. 

It remains to prove Part \ref{thm:main-1} (from distortion along paths to analytic distortion). We know from Theorem \ref{thm:main-MR23} that $f$ is open and discrete. By Corollary \ref{cor:rect_dis}, analytic distortion is controlled by distortion along paths in $X'$. Therefore, the proof of Theorem \ref{thm:main} is complete after we have established the following result. 

\begin{proposition}\label{prop:J_f=0-implies-lower-grad-0}
    If $f \in \operatorname{FDP}(X,\R^2)$ and $K_f\in\Lloc^1(X)$, then $\rho_f^l(x)=0$ (and therefore $\rho_f^u(x)=0$) for almost every $x \in X_0$. 
\end{proposition}

\subsection{Vanishing lower gradient}
This section is devoted to proving Proposition \ref{prop:J_f=0-implies-lower-grad-0}. Let $f\in \operatorname{FDP}(X,\R^2)$ with $K_f\in\Lloc^1(X)$. Towards a contradiction we assume that there exists a set $A\subset X_0$ of positive measure such that $\rho_f^l(x)>0$ for every $x\in A$.

\begin{lemma}\label{lem:good-path}
    There exists a set $A''\subset A$, $\hm^2_X(A\setminus A'')=0$, such that for every $x\in A''$ we find a rectifiable curve  $\gamma_x\colon[0,\ell(\gamma_x)]\to X$ parametrized by arclength and such that
    \begin{enumerate}[label=(\roman*)]
        \item \label{item-1:good-path} the lower gradient inequality \eqref{ineq:lower_gradient} holds for the pair $(f,\rho_f^l)$ on $\gamma_x$,
        \item \label{item-2:good-path} $f$ is absolutely continuous along $\gamma_x$, and
        \item \label{item-3:good-path} there is $0< t <\ell(\gamma_x)$ such that $\gamma_x(t)=x$ and $f\circ\gamma_x$ is differentiable at $t$ with $(f\circ\gamma_x)'(t)>0$. 
    \end{enumerate}
\end{lemma}
\begin{proof} 
    Denote by $\Gamma$ the family of all compact rectifiable curves in $X$ and by $\Gamma_0$ the family of curves $\gamma\in\Gamma$ such that either $\rho_f^u$ is not integrable on $\gamma$ or the upper gradient inequality \eqref{ineq:upper_gradient} does not hold for the pair $(f,\rho_f^u)$ along $\gamma$. Note that $f$ is absolutely continuous along every $\gamma\in\Gamma\setminus\Gamma_0$ and $\mod\Gamma_0=0$, see \cite{HKST:15}*{Propositions 6.3.1 and 6.3.2}. As $0\leq\rho_f^l(x)\leq\rho_f^u(x)$ for almost every $x\in X$, we have that $\rho_f^l$ is integrable on every $\gamma\in\Gamma\setminus\Gamma_0$. Let $\Gamma_1$ be the family of curves $\gamma\in\Gamma\setminus\Gamma_0$ such that the lower gradient inequality \eqref{ineq:lower_gradient} does not hold for the pair $(f,\rho_f^l)$ along $\gamma$. As $f$ is absolutely continuous along every $\gamma\in\Gamma_1$, we have by definition of $\rho_f^l$ that $\mod(\Gamma_1)=0$. 

    Now the claim follows if for almost every $x \in A$ there is a $\gamma_x\colon[0,\ell(\gamma_x)]\to X$ in $\Gamma'=\Gamma\setminus(\Gamma_0\cup\Gamma_1)$ parametrized by arclength and satisfying (iii). 

    Suppose towards contradiction that there is a set $A_0 \subset A$ with positive measure so that, given $x \in A_0$, no $\gamma=\gamma_x \in \Gamma'$ satisfies (iii). Recall that if $\gamma \in \Gamma'$ then $f\circ \gamma$ is differentiable at almost every $0<t<\ell(\gamma)$ (see e.g.\ \cite{HKST:15}*{Remark~4.4.10}). But then, by the definition of the line integral, every $\gamma \in \Gamma'$ satisfies 
        \begin{align*}
        \int_{f\circ\gamma}\chi_{f(A_0)}\,ds=\int_0^{\ell(\gamma)}\chi_{A_0}(\gamma(t))\cdot (f\circ\gamma)'(t)\,dt=0
    \end{align*}
    and therefore
    $$\ell(f\circ\gamma)=\int_{f\circ\gamma}\chi_{f(X\setminus A_0)}\,ds.$$
    Moreover, upper gradient inequality \eqref{ineq:upper_gradient} implies
    $$\int_{f\circ\gamma}\chi_{f(X\setminus A_0)}\,ds\leq\int_\gamma\chi_{X\setminus A_0}\cdot\rho_f^u\,ds.$$
    In particular, $\chi_{X\setminus A_0}\cdot\rho_f^u$ is a weak upper gradient of $f$. From minimality of $\rho_f^u$ we conclude that $\rho_f^u(x)=0$ for almost every $x \in A_0$, which is a contradiction. The proof is complete. 
\end{proof}

Recall that $f$ is discrete and open by Theorem \ref{thm:main-MR23}, so that $f$ is locally invertible outside a countable \emph{branch set}  $\mathcal{B}_f$. We denote $A'=A'' \setminus \mathcal{B}_f$. 

\begin{corollary}\label{cor:length-f(gamma-x)} Fix $x \in A'$, $\gamma_x$, and $0<t<\ell(\gamma_x)$ as in Lemma \ref{lem:good-path}. There are $0<\delta_x,\varepsilon_x<1$ such that if $0<R\leq\delta_x$ and 
$\gamma_R:=\gamma_x|_{[t-R,t+R]}$, then $\ell(\gamma_R)=2R$ and  
\begin{eqnarray*}
diam(|f\circ\gamma_R|)\geq \varepsilon_x R.  
\end{eqnarray*} 
 Moreover, $|\gamma_R|$ has a neighborhood $W$ so that the restriction of $f$ to $W$ is a homeomorphism onto $B(f(x),10(R+\diam(|f \circ \gamma_R|))$. 
\end{corollary}
\begin{proof}
     We set $$\varepsilon_x:=\frac{(f\circ\gamma_x)'(t)}{2}>0.$$ By definition of the metric derivative, we find $0<\delta<t$ such that 
    \begin{align*}
        \frac{d(f(\gamma_x(t-R)),f(\gamma_x(t+R)))}{2R}\geq(f\circ\gamma_x)'(t)-\varepsilon_x=\varepsilon_x 
         \end{align*}
    for every $0<R\leq\delta$. In particular, 
    \begin{align*}
        \diam(|f\circ\gamma_R|)\geq d(f(\gamma_x(t-R)),f(\gamma_x(t+R)))\geq 2\varepsilon_xR
    \end{align*}
    and, as $\gamma_x$ is parametrized by arclength, $\ell(\gamma_R)=2R$. By local invertibility of $f$ at $x$, we may choose $\delta$ to be smaller if necessary so that the last claim also holds. 
\end{proof}

As $\rho_f^l(x)>0$ for almost every $x\in A'$ and $\hm^2(A')>0$, there is  $\varepsilon>0$ such that $\hm^2(A_\varepsilon)>0$ for
$$A_\varepsilon=\{x\in A':\varepsilon_x\geq\varepsilon\}.$$

\begin{proposition}\label{prop:limit-measure}
    We have $J_f(x)>0$ for almost every $x \in A_\varepsilon$. 
\end{proposition}
Proposition \ref{prop:limit-measure} contradicts Lemma \ref{lem:J_f=0}, so Proposition \ref{prop:J_f=0-implies-lower-grad-0}, and thus Theorem \ref{thm:main}, follow once we have proved Proposition \ref{prop:limit-measure}. 

To start the proof of Proposition \ref{prop:limit-measure} we fix $x\in A_\varepsilon \setminus E$, where $E$ is the null set in Theorem \ref{thm:regular}. Then there is $r_x>0$ so that
\begin{equation} \label{eq:hubound}
\hm^2(B(x,r)) \leq 4 r^2 \quad \text{for all } 0< r < r_x. 
\end{equation}

Let $\delta_x,\varepsilon_x>0$ be as in Corollary \ref{cor:length-f(gamma-x)}. We fix a large number $M$ to be specified later, and let $R>0$ be small enough so that 
\begin{equation}\label{mat}
5MR<\min\{r_x,\delta_x\}. 
\end{equation} 
Consider the curve $\gamma_R$ in Corollary \ref{cor:length-f(gamma-x)}. We have  
$$ 
x\in|\gamma_R|\subset B(x,R) \quad \text{and} \quad \diam(f(|\gamma_R|))\geq\varepsilon R. 
$$ 
Without loss of generality we assume that the points $(0,0)$ and $(0,\varepsilon R)$ are contained in $f(|\gamma_R|)$. Let $\pi_2\colon\R^2\to\R$ be the projection to the second coordinate. Then $v=\pi_2\circ f\in\Nloc(X,\R)$. 

Recall that $f$ is invertible in a neighborhood $W$ of $x$ with image 
$$ 
f(W)=B(f(x),10(R+\diam(|f \circ \gamma_R|)). 
$$ 
In particular, for every 
$0<t<\varepsilon R$ there are $s_t$ so that $(s_t,t) \in f(|\gamma_R|)$ and continuum $\eta'_t \subset v^{-1}(t)$ so that 
$$
\eta_t' \cap |\gamma_R| \neq \emptyset \quad \text{and} \quad f(\eta_t')=
I_t:=[s_t-R,s_t+R] \times \{t\}. 
$$ 
We fix $a_t \in \eta_t' \cap |\gamma_R| \subset B(x,R)$ and define
$$
E_M(R)=\{0<t<\varepsilon R:\eta'_t\not\subset B(x,MR)\}.
$$ 
We may choose continua $\eta_t'$ so that $E_M(R)$ is a Borel set. 

\begin{lemma}\label{lem:E_M}
For almost every $x \in A_\varepsilon$ we can choose $M$ (depending on $x$) so that  
$$
|E_M(R)|_1\leq\frac{\varepsilon R}{2}
$$ 
for all $R>0$ satisfying \eqref{mat}. 
\end{lemma}

\begin{proof}
    We may assume that $|E_M(R)|_1>0$ since otherwise there is nothing to show. We may also assume that $M=2^l$ for some $l\in\N$. Define $\phi\colon X\to\R$ by
    $$
    \phi(y)=\frac{\chi_{B(x,MR)\setminus B(x,R)}(y)}{l\cdot d(y,x)}. 
    $$
    Let $\eta_t$ be the $a_t$-component of $\eta'_t\cap B(x,MR)$, and 
    $$
    \eta_t^j=\eta_t\cap B(x,2^jR)\setminus B(x,2^{j-1}R). 
    $$ 
    If $t\in E_M(R)$, then
    \begin{align}
      \nonumber  \int_{\eta_t}\phi\,d\hm^1&=\sum_{j=1}^l\int_{\eta_t^j}\phi\,d\hm^1\geq\frac{1}{l}\sum_{j=1}^l\hm^1(\eta_t^j)\,\min_{y\in\eta_t^j}\frac{1}{d(y,x)}\\ \label{logar}
        &\geq\frac{1}{l}\sum_{j=1}^l(2^{j-1}R)\cdot\frac{1}{2^jR}\geq\frac{1}{2}.
    \end{align}
    Note that each $\eta_t$ is a non-degenerate continuum contained in the level set $v^{-1}(t)$. Hence, by \cite{MR23}*{Lemma~4.4}, $\hm^1(\eta_t\cap X_0)=0$ for almost every $t\in E_M(R)$. Let 
    $$
    F_M=\bigcup_{t\in E_M(R)}\eta_t\subset B(x,MR).  
    $$ 
    We apply \eqref{logar}, the coarea inequality for Sobolev mappings \cite{MN23}*{Theorem 1.6}, and Hölder's inequality to obtain 
    \begin{align} 
    \begin{split} \label{eq:lem-1}
        \frac{|E_M(R)|_1}{2}&\leq\int_{E_M(R)}\int_{\eta_t}\phi\,d\hm^1\,ds=\int_{E_M(R)}\int_{\eta_t\cap X'}\phi\,d\hm^1\,ds\\ 
        &\leq\frac{4}{\pi}\int_{F_M\cap X'}\phi\,\rho_f^u\,d\hm^2\leq\frac{4}{\pi}\int_{F_M\cap X'}\phi\,K_f^{1/2}(\rho_f^u\rho_f^l)^{1/2}\,d\hm^2\\
        &\leq\frac{4}{\pi}\biggl(\underbrace{\int_{B(x,MR)\setminus B(x,R)}\phi^2K_f\,d\hm^2}_{=:\mathcal{I}_1}\biggr)^{1/2}\biggl(\underbrace{\int_{F_M\cap X'}\rho_f^u\rho_f^l\,d\hm^2}_{=:\mathcal{I}_2}\biggr)^{1/2}. 
    \end{split}
    \end{align}
    Recall that $K_f\in\Lloc^1(X)$ and therefore, by Lemma \ref{lem:max-fct}, for almost every $x \in A_\varepsilon \setminus E$ there exists $L\geq1$ such that the maximal function satisfies $\mathcal{M}K_f(x)\leq L$. Combining with \eqref{eq:hubound}, we obtain 
    \begin{align}
        \begin{split}\label{eq:lem-2}
            \mathcal{I}_1&\leq\frac{4}{l^2R^2}\sum_{j=1}^l2^{-2j}\int_{B(x,2^jR)\setminus B(x,2^{j-1}R)}K_f\,d\hm^2\\
            &\leq\frac{4}{l^2R^2}\sum_{j=1}^l2^{-2j}\frac{\hm^2(B(x,5\cdot2^jR))}{\hm^2(B(x,5\cdot 2^jR))}\int_{B(x,2^jR)}K_f\,d\hm^2\\
            &\leq\frac{400}{l^2}\sum_{j=1}^l\frac{1}{\hm^2(B(x,5\cdot 2^jR))}\int_{B(x,2^jR)}K_f\,d\hm^2\\
            &\leq\frac{400Ll}{l^2}=\frac{400L}{l}. 
        \end{split}
    \end{align}
    We may apply the area inequality (Theorem \ref{thm:area-ineq}) and Fubini's theorem to compute $\mathcal{I}_2$ as follows: 
    \begin{align}
        \begin{split}\label{eq:lem-3}
            \mathcal{I}_2&\leq4\sqrt{2}\int_{f(F_M)}1\,dA\leq8\sqrt{2}R\,|E_M(R)|_1,
        \end{split}
    \end{align}
    where the last inequality holds as $f(\eta_t)\subset I_t$, $I_t$ is a segment of length $2R$ for every $t\in E_M(R)$, and $f(F_M)=\bigcup_{t\in E_M(R)}f(\eta_t)$. Combining \eqref{eq:lem-1}, \eqref{eq:lem-2} and \eqref{eq:lem-3} gives
    \begin{align*}
        \frac{|E_M(R)|_1}{2}\leq\left(3200\sqrt{2}L\,\frac{R}{\ell}\,|E_M(R)|_1\right)^{1/2}.
    \end{align*}
    After setting $\kappa=50000L$ we obtain
    $$|E_M(R)|_1\leq\kappa\frac{R}{l}, $$
    and thus $|E_M(R)|_1\leq\frac{\varepsilon R}{2}$ for $l$ large enough.
\end{proof}

We now finish the proof of Proposition \ref{prop:limit-measure}. Choose 
$x \in A_\varepsilon$, $M$ and $R$ so that Lemma \ref{lem:E_M} holds. We denote $Q_M(R)=(0,\varepsilon R) \setminus E_M(R)$. By Lemma \ref{lem:E_M}, 
$$
|Q_M(R)|_1 \geq \frac{\varepsilon R}{2}. 
$$
Moreover, local invertibility of $f$ around $x$ and the definition of $Q_M(R)$ show that if $t \in Q_M(R)$ then $f(\eta_t')=I_t$, so that $|f(\eta_t')|_1=2R$. We denote 
$$
G_M(R)=\bigcup_{t\in Q_M(R)}f(\eta_t'). 
$$ 
Note that by definition, $G_M(R)\subset f(B(x,MR))$. Fubini's theorem now gives 
$$ 
\varepsilon R^2 \leq 2R \cdot |Q_M(R)|_1= |G_M(R)|_2 \leq |f(B(x,MR))|_2. 
$$
Proposition \ref{prop:limit-measure} follows by letting $R\to 0$. The proof of Theorem \ref{thm:main} is complete. 

\section{Quasiconformal uniformization}\label{sec:reciprocal}

This section is devoted to proving Corollary \ref{cor:qc-equivalence} and Theorems \ref{thm:reciprocal} and \ref{thm:factorization}. Before proving Corollary \ref{cor:qc-equivalence}, we recall the following theorem of Williams (\cite{Wil:12}*{Theorem 1.1}).
\begin{thm}\label{thm:williams}
Let $f\colon X\to Y$ a homeomorphism between metric surfaces. Then the following conditions are equivalent with the same constant $C\geq1$.
\begin{enumerate}[label=\normalfont(\roman*)]
    \item \label{williams:ana} $f\in\Nloc(X,Y)$ and 
    \begin{align*}
        \rho_f^u(x)^2\leq C\cdot\jac_f(x)\quad\text{for almost every }x\in X,
    \end{align*}
    where $\jac_f$ denotes the Radon-Nikodym derivative of the corresponding pull-back measure with respect to $\hm^2_X$.
    \item For every family $\Gamma$ of curves in $X$ we have
    \begin{align}\label{williams:geo}
        \mod \Gamma \leq C\cdot\mod f(\Gamma).
    \end{align}
\end{enumerate}
\end{thm}

\begin{proof}[Proof of Corollary \ref{cor:qc-equivalence}]
    We first note that Corollary \ref{cor:jacohomeo} implies the equivalence of analytic quasiconformality as defined in our work and Condition \ref{williams:ana} in Theorem \ref{thm:williams}. In particular, the constant may be chosen to be the same.

    In the next step we show that within this setting $f$ satisfying \eqref{williams:geo} for some $C\geq1$ implies that $f$ is geometrically $4C$-quasiconformal. Namely, as $f$ maps into $\R^2$ and satisfies \eqref{williams:geo}, it follows from \cite{Raj:17} (see also \cite{NR22}, \cite{RR19}) that there exists a geometrically 2-quasiconformal homeomorphism $u\colon U\to X$, where $U\subset\R^2$ is a domain. Now the map $h:=f\circ u\colon U\to\R^2$ is a homeomorphism satisfying 
    $\mod \Gamma\leq2C\cdot\mod h(\Gamma)$
    for every family $\Gamma$ of curves in $X$. As the domains are planar, $h$ is geometrically $2C$-quasiconformal. Thus, $f$ is geometrically $4C$-quasiconformal. This shows that \eqref{williams:geo} is quantitatively equivalent to geometric quasiconformality.

    Theorem \ref{thm:williams} now implies the equivalence between Conditions \eqref{cor:qc-eq-ana} and \eqref{cor:qc-eq-geo} in Corollary \ref{cor:qc-equivalence}, i.e., between analytic and geometric quasiconformality. Moreover, it follows from our main theorem (Theorem \ref{thm:main}) that \eqref{cor:qc-eq-ana} is quantitatively equivalent with \eqref{cor:qc-eq-path}. More explicitly, if $f$ is analytically $C$-qua\-si\-con\-for\-mal then $f$ is $K$-quasiconformal along paths with $K=4\sqrt{2}C$ and vice versa. We have proven the equivalence of Conditions (1)-(3). 

    It remains to show that if $f$ is quasiconformal according to any of the conditions above, then so is $f^{-1}$. First, it follows from the definition of geometric quasiconformality that $f^{-1}$ is geometrically quasiconformal. Moreover, applying Theorem \ref{thm:williams} to $f^{-1}$ shows that $f^{-1}$ is also analytically quasiconformal. In particular,  $f^{-1} \in \Nloc(f(X),X)$. Quasiconformality of $f^{-1}$ along paths now follows from analytic quasiconformality by combining Theorem \ref{thm:classical-area-formula} and Lemma \ref{lem:MR23-2.9}. 
    The proof is complete. \end{proof}

\begin{proof}[Proof of Theorem \ref{thm:reciprocal}]
    Let $f\in\operatorname{FDP}(X,\R^2)$ be $K$-quasiregular. By Theorem \ref{thm:unif} (iii), there is a weakly quasiconformal parametrization $u:U \to X$, where $U\subset \R^2$, so that $h=f\circ u$ is in $\operatorname{FDP}(U,\R^2)$ and $\sqrt{2}K$-quasiregular. By Theorem \ref{thm:main-MR23}, $h$ is discrete and open. We conclude that $u$ is both discrete and monotone and therefore a homeomorphism. We will show that $\phi:=u^{-1}\colon X\to U$ is (geometrically) quasiconformal. 
   
    Denote by $\mathcal{B}_f$ the set of branch points of $f$, i.e., the set of points at which $f$ is not locally invertible, and recall that $\mathcal{B}_f$ is a discrete set. For every $x\in X\setminus\mathcal{B}_f$ we find a neighbourhood $V_x\subset X$ of $x$ such that $f|_{V_x}$ is a homeomorphism onto its image. It follows from the proof of Corollary \ref{cor:qc-equivalence} that $f|_{V_x}$ is geometrically $16\sqrt{2}K$-quasiconformal and $h|_{u^{-1}(V_x)}$ is geometrically $32K$-quasiconformal. In particular, $\phi|_{V_x}$ is geometrically $C$-quasiconformal with $C=512\sqrt{2}K$.

    The proof of Corollary \ref{cor:qc-equivalence} implies that the restriction of $\phi$ to $\tilde{X}=X \setminus \mathcal{B}_f$ is analytically $C$-quasiconformal. In particular, there is a family $\Gamma_0$ of curves in $\tilde{X}$ with zero modulus so that $\phi$ is absolutely continuous on all paths $\tilde{\gamma} \notin \Gamma_0$ in $\tilde{X}$. Theorem \ref{thm:reciprocal} follows if we can show that $\phi$ is absolutely continuous on almost every rectifiable curve $\gamma$ in $X$. 

    We fix such a $\gamma$. Then, since $\mathcal{B}_f$ is discrete, $|\gamma| \cap \mathcal{B}_f$ is finite. Since $\phi$ is continuous, we conclude that if $\phi$ is not absolutely continuous on $\gamma$ then there is a subpath $\tilde{\gamma}$ with image in $\tilde{X}$ where $\phi$ is not absolutely continuous, that is, $\tilde{\gamma} \in \Gamma_0$. The family of paths in $X$ which contain a subpath in $\Gamma_0$ has zero modulus. The proof is complete. 
\end{proof}

\begin{proof}[Proof of Theorem \ref{thm:factorization}]
    By Theorem \ref{thm:reciprocal} there is a geometrically quasiconformal homeomorphism $\phi\colon X\to U$, where $U$ is a domain in $\R^2$. The map $h=f\circ \phi^{-1}:U \to \R^2$ is quasiregular. By the measurable Riemann mapping theorem (see e.g.\ \cite{AIM09}*{Theorem 5.3.4}) there exists a quasiconformal map $\psi\colon U\to V$, $V\subset\R^2$, such that $g:=h\circ\psi^{-1}:V \to \R^2$ is analytic. The statement follows after setting $v:=\psi\circ \phi$. 
\end{proof}



\begin{bibdiv}
\begin{biblist}

\bib{AIM09}{book}{
      author={Astala, Kari},
      author={Iwaniec, Tadeusz},
      author={Martin, Gaven},
       title={Elliptic partial differential equations and quasiconformal
  mappings in the plane},
      series={Princeton Mathematical Series},
   publisher={Princeton University Press, Princeton, NJ},
        date={2009},
      volume={48},
        ISBN={978-0-691-13777-3},
      review={\MR{2472875}},
}

\bib{AmbTil04}{book}{
      author={Ambrosio, Luigi},
      author={Tilli, Paolo},
       title={Topics on analysis in metric spaces},
      series={Oxford Lecture Series in Mathematics and its Applications},
   publisher={Oxford University Press, Oxford},
        date={2004},
      volume={25},
        ISBN={0-19-852938-4},
      review={\MR{2039660}},
}

\bib{Bal97}{incollection}{
      author={Ball, Keith},
       title={An elementary introduction to modern convex geometry},
        date={1997},
   booktitle={Flavors of geometry},
      series={Math. Sci. Res. Inst. Publ.},
      volume={31},
   publisher={Cambridge Univ. Press, Cambridge},
       pages={1\ndash 58},
}

\bib{BonKle02}{article}{
      author={Bonk, Mario},
      author={Kleiner, Bruce},
       title={Quasisymmetric parametrizations of two-dimensional metric
  spheres},
        date={2002},
        ISSN={0020-9910},
     journal={Invent. Math.},
      volume={150},
      number={1},
       pages={127\ndash 183},
         url={https://doi.org/10.1007/s00222-002-0233-z},
      review={\MR{1930885}},
}

\bib{BKR07}{article}{
      author={Balogh, Zolt\'{a}n~M.},
      author={Koskela, Pekka},
      author={Rogovin, Sari},
       title={Absolute continuity of quasiconformal mappings on curves},
        date={2007},
        ISSN={1016-443X},
     journal={Geom. Funct. Anal.},
      volume={17},
      number={3},
       pages={645\ndash 664},
         url={https://doi.org/10.1007/s00039-007-0607-x},
      review={\MR{2346270}},
}

\bib{Cri06}{article}{
      author={Cristea, Mihai},
       title={Quasiregularity in metric spaces},
        date={2006},
        ISSN={0035-3965},
     journal={Rev. Roumaine Math. Pures Appl.},
      volume={51},
      number={3},
       pages={291\ndash 310},
      review={\MR{2275345}},
}

\bib{EG92}{book}{
      author={Evans, Lawrence~C.},
      author={Gariepy, Ronald~F.},
       title={Measure theory and fine properties of functions},
      series={Studies in Advanced Mathematics},
   publisher={CRC Press, Boca Raton, FL},
        date={1992},
}

\bib{Fed69}{book}{
      author={Federer, Herbert},
       title={Geometric measure theory},
      series={Die Grundlehren der mathematischen Wissenschaften},
   publisher={Springer-Verlag New York, Inc., New York},
        date={1969},
      volume={153},
}

\bib{Guo15}{article}{
      author={Guo, Chang-yu},
       title={Mappings of finite distortion between metric measure spaces},
        date={2015},
     journal={Conform. Geom. Dyn.},
      volume={19},
       pages={95\ndash 121},
         url={https://doi.org/10.1090/ecgd/277},
      review={\MR{3338960}},
}

\bib{Hei01}{book}{
      author={Heinonen, Juha},
       title={Lectures on analysis on metric spaces},
      series={Universitext},
   publisher={Springer-Verlag},
     address={New York},
        date={2001},
        ISBN={0-387-95104-0},
         url={http://dx.doi.org/10.1007/978-1-4613-0131-8},
      review={\MR{1800917 (2002c:30028)}},
}

\bib{HeiKos95}{article}{
      author={Heinonen, Juha},
      author={Koskela, Pekka},
       title={Definitions of quasiconformality},
        date={1995},
        ISSN={0020-9910},
     journal={Invent. Math.},
      volume={120},
      number={1},
       pages={61\ndash 79},
         url={https://doi.org/10.1007/BF01241122},
      review={\MR{1323982}},
}

\bib{HeiKos98}{article}{
      author={Heinonen, Juha},
      author={Koskela, Pekka},
       title={Quasiconformal maps in metric spaces with controlled geometry},
        date={1998},
        ISSN={0001-5962},
     journal={Acta Math.},
      volume={181},
      number={1},
       pages={1\ndash 61},
         url={https://doi.org/10.1007/BF02392747},
      review={\MR{1654771}},
}

\bib{HKST01}{article}{
      author={Heinonen, Juha},
      author={Koskela, Pekka},
      author={Shanmugalingam, Nageswari},
      author={Tyson, Jeremy~T.},
       title={Sobolev classes of {B}anach space-valued functions and
  quasiconformal mappings},
        date={2001},
        ISSN={0021-7670},
     journal={J. Anal. Math.},
      volume={85},
       pages={87\ndash 139},
         url={https://doi.org/10.1007/BF02788076},
      review={\MR{1869604}},
}

\bib{HKST:15}{book}{
      author={Heinonen, Juha},
      author={Koskela, Pekka},
      author={Shanmugalingam, Nageswari},
      author={Tyson, Jeremy~T.},
       title={Sobolev spaces on metric measure spaces: An approach based on
  upper gradients},
      series={New Mathematical Monographs},
   publisher={Cambridge University Press, Cambridge},
        date={2015},
      volume={27},
}

\bib{Iko:19}{article}{
      author={Ikonen, Toni},
       title={Uniformization of metric surfaces using isothermal coordinates},
        date={2022},
     journal={Ann. Fenn. Math.},
      volume={47},
      number={1},
       pages={155\ndash 180},
}

\bib{Kar07}{article}{
      author={Karmanova, M.~B.},
       title={Area and co-area formulas for mappings of the {S}obolev classes
  with values in a metric space},
        date={2007},
     journal={Sibirsk. Mat. Zh.},
      volume={48},
      number={4},
       pages={778\ndash 788},
}

\bib{Kir14}{article}{
      author={Kirsil\"{a}, Ville},
       title={Mappings of finite distortion from generalized manifolds},
        date={2014},
     journal={Conform. Geom. Dyn.},
      volume={18},
       pages={229\ndash 262},
         url={https://doi.org/10.1090/S1088-4173-2014-00272-0},
      review={\MR{3278159}},
}

\bib{Kir94}{article}{
      author={Kirchheim, Bernd},
       title={Rectifiable metric spaces: local structure and regularity of the
  {H}ausdorff measure},
        date={1994},
        ISSN={0002-9939},
     journal={Proc. Amer. Math. Soc.},
      volume={121},
      number={1},
       pages={113\ndash 123},
         url={https://doi.org/10.2307/2160371},
      review={\MR{1189747}},
}

\bib{LP20}{article}{
      author={Luisto, Rami},
      author={Pankka, Pekka},
       title={Sto\"{\i}low's theorem revisited},
        date={2020},
        ISSN={0723-0869,1878-0792},
     journal={Expo. Math.},
      volume={38},
      number={3},
       pages={303\ndash 318},
         url={https://doi.org/10.1016/j.exmath.2019.04.002},
      review={\MR{4149174}},
}

\bib{LWarea}{article}{
      author={Lytchak, Alexander},
      author={Wenger, Stefan},
       title={Area minimizing discs in metric spaces},
        date={2017},
     journal={Arch. Ration. Mech. Anal.},
      volume={223},
      number={3},
       pages={1123\ndash 1182},
}

\bib{LWintrinsic}{article}{
      author={Lytchak, Alexander},
      author={Wenger, Stefan},
       title={Intrinsic structure of minimal discs in metric spaces},
        date={2018},
     journal={Geom. Topol.},
      volume={22},
      number={1},
       pages={591\ndash 644},
}

\bib{Mei22}{article}{
      author={Meier, Damaris},
       title={Quasiconformal uniformization of metric surfaces of higher
  topology},
        date={2024},
     journal={Indiana Univ. Math. J.},
        note={To appear},
}

\bib{MN23}{article}{
      author={Meier, Damaris},
      author={Ntalampekos, Dimitrios},
       title={Lipschitz-volume rigidity and {S}obolev coarea inequality for
  metric surfaces},
        date={2024},
        ISSN={1050-6926,1559-002X},
     journal={J. Geom. Anal.},
      volume={34},
      number={5},
       pages={Paper No. 128, 30},
         url={https://doi.org/10.1007/s12220-024-01577-x},
      review={\MR{4718625}},
}

\bib{MR23}{article}{
      author={Meier, Damaris},
      author={Rajala, Kai},
       title={Mappings of finite distortion on metric surfaces},
        date={2023},
       pages={preprint arXiv:2309.15615},
}

\bib{MW21}{article}{
      author={Meier, Damaris},
      author={Wenger, Stefan},
       title={Quasiconformal almost parametrizations of metric surfaces},
        date={2024},
     journal={J. Eur. Math. Soc.},
        note={published online first},
}

\bib{NR22}{article}{
      author={Ntalampekos, Dimitrios},
      author={Romney, Matthew},
       title={Polyhedral approximation and uniformization for non-length
  surfaces},
        date={2022},
       pages={preprint arXiv:2206.01128},
}

\bib{NR:21}{article}{
      author={Ntalampekos, Dimitrios},
      author={Romney, Matthew},
       title={Polyhedral approximation of metric surfaces and applications to
  uniformization},
        date={2023},
        ISSN={0012-7094,1547-7398},
     journal={Duke Math. J.},
      volume={172},
      number={9},
       pages={1673\ndash 1734},
         url={https://doi.org/10.1215/00127094-2022-0061},
      review={\MR{4608329}},
}

\bib{OnnRaj09}{article}{
      author={Onninen, Jani},
      author={Rajala, Kai},
       title={Quasiregular mappings to generalized manifolds},
        date={2009},
        ISSN={0021-7670},
     journal={J. Anal. Math.},
      volume={109},
       pages={33\ndash 79},
         url={https://doi.org/10.1007/s11854-009-0028-x},
      review={\MR{2585391}},
}

\bib{Raj:17}{article}{
      author={Rajala, Kai},
       title={Uniformization of two-dimensional metric surfaces},
        date={2017},
     journal={Invent. Math.},
      volume={207},
      number={3},
       pages={1301\ndash 1375},
}

\bib{Ric93}{book}{
      author={Rickman, Seppo},
       title={Quasiregular mappings},
      series={Ergebnisse der Mathematik und ihrer Grenzgebiete (3) [Results in
  Mathematics and Related Areas (3)]},
   publisher={Springer-Verlag, Berlin},
        date={1993},
      volume={26},
        ISBN={3-540-56648-1},
         url={https://doi.org/10.1007/978-3-642-78201-5},
      review={\MR{1238941}},
}

\bib{RR19}{article}{
      author={Rajala, Kai},
      author={Romney, Matthew},
       title={Reciprocal lower bound on modulus of curve families in metric
  surfaces},
        date={2019},
        ISSN={1239-629X},
     journal={Ann. Acad. Sci. Fenn. Math.},
      volume={44},
      number={2},
       pages={681\ndash 692},
         url={https://doi.org/10.5186/aasfm.2019.4442},
      review={\MR{3973535}},
}

\bib{RRR21}{article}{
      author={Rajala, Kai},
      author={Rasimus, Martti},
      author={Romney, Matthew},
       title={Uniformization with infinitesimally metric measures},
        date={2021},
        ISSN={1050-6926},
     journal={J. Geom. Anal.},
      volume={31},
      number={11},
       pages={11445\ndash 11470},
         url={https://doi.org/10.1007/s12220-021-00689-y},
      review={\MR{4310179}},
}

\bib{Tys01}{article}{
      author={Tyson, Jeremy~T.},
       title={Metric and geometric quasiconformality in {A}hlfors regular
  {L}oewner spaces},
        date={2001},
     journal={Conform. Geom. Dyn.},
      volume={5},
       pages={21\ndash 73},
         url={https://doi.org/10.1090/S1088-4173-01-00064-9},
      review={\MR{1872156}},
}

\bib{Tys98}{article}{
      author={Tyson, Jeremy},
       title={Quasiconformality and quasisymmetry in metric measure spaces},
        date={1998},
        ISSN={1239-629X},
     journal={Ann. Acad. Sci. Fenn. Math.},
      volume={23},
      number={2},
       pages={525\ndash 548},
      review={\MR{1642158}},
}

\bib{Wil:12}{article}{
      author={Williams, Marshall},
       title={Geometric and analytic quasiconformality in metric measure
  spaces},
        date={2012},
     journal={Proc. Amer. Math. Soc.},
      volume={140},
      number={4},
       pages={1251\ndash 1266},
}

\end{biblist}
\end{bibdiv}
\end{document}